\RequirePackage{tikz}
\documentclass{article}
\pdfoutput=1
\newif\iflong
\longtrue % for version on the Web
\usepackage[utf8]{inputenc}
\usepackage{amsmath}
\usepackage{amssymb}
\usepackage{amsthm}
\usepackage{amsfonts}
\usepackage{tikz,tikz-3dplot}
\usepackage{epsfig}
\usepackage{url}
\usepackage[all,knot,poly]{xy}
\tdplotsetmaincoords{70}{115}
\usepackage{todonotes}
\usetikzlibrary{matrix}
\def\R{\mbox{$\mathbb{R}$}}
\newtheorem{definition}{Definition}[section]
\newtheorem{lemma}[definition]{Lemma} %[section]
\newtheorem{remark}[definition]{Remark} %[section]
\newtheorem{example}[definition]{Example} %[section]
\usepackage{tikz-cd}

\newcommand\look[1]{{\bf #1}}

\newcommand\ForAuthors[1]%          %  temporary remark for the
 {\par\smallskip                     %  authors:
  \begin{center}%                    %
   \fbox%                            %    --------
   {\parbox{0.9\linewidth}%          %    |  #1  |
    {\raggedright--- #1}%         %    --------
   }%                                %
  \end{center}%                      %
  \par\smallskip                     %
 }

\tikzset{
  curarrow/.style={
  rounded corners=8pt,
  execute at begin to={every node/.style={fill=red}},
    to path={-- ([xshift=-50pt]\tikztostart.center)
    |- (#1) node[fill=white] {$\scriptstyle \partial_*$}
    -| ([xshift=50pt]\tikztotarget.center)
    -- (\tikztotarget)}
    }
}
\newcounter{nodemaker}
\definecolor{mred}{rgb}{0.7,0.1,0.1}
\definecolor{mblue}{rgb}{0,0,0.8}
\definecolor{mgreen}{rgb}{0,0.6,0.3}

\newcommand\N[0]{\mathbb{N}}

\def\Z{\mbox{$\mathbb{Z}$}}

\newcommand\functor[1][l]{\csname#1functor\endcsname}
\newcommand\lfunctor[3]{%
  \setbox0=\hbox{$#2$}%
  \kern\wd0%
  \ensurestackMath{\Centerstack[c]{#1\\ \mathllap{#2\;\,}\mathclap{\DownArrow}\\#3}}%
}
\newcommand\rfunctor[3]{%
  \setbox0=\hbox{$#2$}%
  \ensurestackMath{\Centerstack[c]{#1\\\mathclap{\DownArrow}\mathrlap{\,\;#2}\\#3}}%
  \kern\wd0%
}

\newcommand\DownArrow{\rotatebox[origin=c]{-90}{$\longrightarrow$\,}}

\title{Directed homological and cohomological operations}
\author{Eric Goubault }
\date{\today}

\begin{document}

\maketitle

\begin{abstract}
In this short note, we present a persistence module approach to directed cohomology, dual to the directed homology introduced in \cite{directedpersistencemodule}. We lay out the first properties of directed cohomology and in particular of cohomological operations, partially linked to some homological operations. We treat here both the case of a specific class of precubical sets and of general directed spaces.
%as well as the case of general directed spaces.  
\end{abstract}

\section{Introduction}

Directed homology has been long studied \cite{concur92,directedpersistencemodule,grandis2004,fahrenberg2004directed,patchkoria2006,dubut2017} but very few has been developed concerning homological operations and cohomological operations. In a classical setup, cap and cup products play an important role, giving finer control on these topological invariants. 

We build on the directed homology introduced in \cite{directedpersistencemodule}, to derive two main operations. One is induced by the concatenation of directed paths and higher paths, and leads both to a homological and cohomological operation. The other one is induced by the classical cup product on the ``local cohomology ring" of the traces spaces, and lives only in directed cohomology. 

In this short note, we pave the way towards the full study of these operations. More has to be studied, in particular the interaction between the two directed cohomological operations and on practical calculations. 

\paragraph{Contents}

We first recap in Section \ref{sec:backgrounddspace} and in Section 
\ref{sec:background} the basic concepts and definitions we will need concerning directed spaces and their combinatorial counterparts, a certain class of precubical sets. 

We then define coboundary maps in 
Section \ref{sec:coboundarymaps}, by dualizing the construction of \cite{directedpersistencemodule}, which allows us to defined directed cohomology bimodules, over the path algebra in 
Section \ref{sec:precubcohomology}. On a particular class of precubical sets, we are able to construct in Section \ref{sec:directedoperationscub}
two cohomological operations: $\curvearrowright$ and $\smile$, $\curvearrowright$ being the dual of a homology operation $\circledast$ generated by the concatenation operation of cube chains. 

We describe more briefly the corresponding construction in the more general situation of directed spaces, see 
Section \ref{sec:dspace}. 

Finally we give some calculations of these operations on directed cohomological bimodules in 
Section \ref{sec:examples}, building from the results of \cite{Barychnikov}.

%(also feasible for d-spaces)

\section{Background on directed spaces}

\label{sec:backgrounddspace}
 We first recap the notion of $d$-space \cite{thebook,grandisbook}: % which is one of the most general models for directed topology: 

\begin{definition}
A {directed space} is a pair $\Vec{X} = (X, dX)$ where $X$ is a topological space and $dX$ is a set of continuous paths in $X$, called the \textbf{directed paths} %(or \textbf{dipaths})
of the space, such that;
\begin{enumerate}
    \item[(i)] Every constant path is a directed path
    \item[(ii)] Every increasing reparametrization of a directed path is a directed path i.e. for every directed path $\gamma \in dX$ and for any function $\varphi : [0,1] \to [0,1]$ which is continuous and increasing, $\gamma \circ \varphi \in dX$  
    \item[(iii)] For any directed paths $\gamma$ and $\delta$ so that $\gamma(1) = \delta(0)$, the concatenation $\gamma \cdot \delta$ of $\gamma$ and $\delta$ is a directed path
\end{enumerate}
\end{definition}

One simple example of a directed space is the closed interval $[0,1]$ with increasing function $[0,1] \to [0,1]$ as directed paths. This space is the {directed unit interval}.

Morphisms of directed spaces are defined as follows.
\begin{definition}
Let $\Vec{X} = (X, dX)$ and $\Vec{Y} = (Y, dY)$ be two directed spaces. A continuous function $f : X \to Y$ is called a {directed map}, or {dimap}, when:
\[\forall \gamma \in dX, \quad 
f \circ \gamma \in dY\]

\end{definition}

A dimap from the directed unit interval to a directed space $\Vec{X}$ is called a {dipath}. The dipaths of $\Vec{X}$ are exactly its directed paths.

Directed spaces together with dimaps (and composition of functions) form the category of directed spaces. %An isomorphism in this category is called a \textbf{dihomeomorphism}. 

In a given directed space, the set of paths from one point to another is endowed with the compact-open topology.

\begin{definition}
Let $\Vec{X} = (X, dX)$ be a directed space and $x,y \in X$. 

The set of all directed paths $\gamma \in dX$ so that 
\[\gamma(0) = x \textit{ and } \gamma (1) = y\]
together with the compact-open topology is called the {path space} from $x$ to $y$ and is denoted by ${P(\Vec{X})_x^y}$ 
\end{definition}

We consider the relation $\sim$ ``being the same up to bijective increasing repara\-metrization'' defined by
$\gamma \sim \delta$  iff there exists a continuous increasing bijective function 
\[f : [0,1] \to [0,1]\]
\noindent such that
\[\delta = \gamma \circ f \]
is indeed an equivalence relation.

\begin{definition}
Given a directed space $\Vec{X} = (X, dX)$, a {trace} from $x \in X$ to $y \in X$ is an equivalence class of paths from $x$ to $y$ for the equivalence relation here denoted by $\sim$.
If $\gamma$ is a dipath, we will denote by ${[\gamma]}$ the corresponding trace.
\end{definition}

This allows us to define concatenation $*$ on traces, and this gives rise to a category of traces. We denote by $Tr(\Vec{X})$ the trace space of $\Vec{X}$, and $Tr(\Vec{X})_\alpha^\beta$ the subspace of $Tr(\Vec{X})$ consisting of traces from points $\alpha$ to $\beta$ in $X$. 

\section{Background on precubical sets and cube chains}

\label{sec:background}

Our results will apply to a particular class of precubical sets, that has been introduced in \cite{cubechains}: 

\begin{definition}
Let $X$ be a finite precubical set. We say that $X$ is \look{proper} if the map: 
$$
\bigcup\limits_{n\geq 0} X_n : \ c \rightarrow \{d^0(c),d^1(c)\} \in 2^{X_0}$$
\noindent is an injection, i.e., the cubes of $X$ can be distinguished by its start and end vertices. 

The \look{non-looping length covering} of $X$ is the precubical set $\tilde{X}$ with $\tilde{X}_n=X_n\times \Z$ and $d^\epsilon_i(c,k)=(d^\epsilon_i(c),k+\epsilon)$.
\end{definition}

We write ${Cub}$ for the full sub-category of $Precub$, of finite precubical sets with proper non-looping length covering. 

%\todo[inline]{Mettre $X\in {Cub}$ etc. dans tout le reste a la place de "finite" etc.}

The following definition is essential in the calculations on trace spaces done by Krzysztof Ziemianski \cite{cubechains}. We twist a little the original definition so that 1-cube chains can be empty, and can consist in that case of constant paths on vertices, this convention will prove easier for the rest of the paper. 

\begin{definition}
Let $X$ be a pre-cubical set and let $v$, $w \in X_0$ be two of its vertices. A \look{cube chain} in $X$ from $v$ to $w$ is a sequence of cubes indexed by $v$ and $w$, $c = (c_1,\ldots,c_l)_{v,w}$, where $c_k \in X_{n_k}$ and $n_k > 0$, $l\geq 0$, such that either $l=0$ and $v=w$ or: 
\begin{itemize}
    \item $d^0(c_1) = v$,
\item $d^1(c_l)=w$,
\item $d^1(c_i) = d^0(c_{i+1})$ for $i = 1,\ldots, l - 1$.
\end{itemize}
We will often leave out the index $v,w$ to cube chains when the context makes it clear what they are. 
\end{definition}

The sequence $(n_1,\ldots,n_l)$ will be called the \look{type of a cube chain} $c$, $dim(c) = n_1 + \ldots + n_l - l$ the \look{dimension} of $c$, and $n_1 + \ldots + n_l$, the \look{length} of $c$. By convention, we set $dim(()_{v,v})=0$: the dimension of constant paths on a vertex is zero. These cube chains in $X$ from $v$ to $w$ will be denoted by $Ch(X)^w_v$, and the set of cube chains of dimension equal to $m$ (resp. less than $m$, less or equal to $m$) by $Ch^{=m}(X)^w_v$ (resp. $Ch^{<m}(X)^w_v$, $Ch^{\leq m}(X)^w_v$). Note that a cube chain has dimension 0 if and only if it contains 1-cubes only, or is an empty cube chain. 

\begin{remark}
\label{rem:0cubechain}
Cube chains of dimension 0 are naturally identified with directed paths (including the constant paths on vertices) of the underlying quiver of the pre-cubical set. 
\end{remark}

\begin{remark}
\label{rem:1cubechain}
1-cube chains are necessarily of the form $(c_1,\ldots,c_l)$ with a unique $c_i$ of dimension 2, all the others being of dimension 1. Indeed, each cell of dimension strictly greater than 1 contributes to adding to the dimension of the cube chain. 
\end{remark}

\begin{remark}
\label{rem:dimension}
We note that the dimension of cube chains is additive in the following sense: let $c=(c_1,\ldots,c_l)$ and $d=(d_1,\ldots,d_k)$ be two cube chains. Their concatenation, when defined, i.e. when $d^1(c_l)=d^0(d_1)$, is the cube chain $e=(c_1,\ldots,c_l,d_1,\ldots,d_k)$. The dimension of $e$ is the sum of the dimensions of each cell $c_i$ and $d_j$ minus $k+l$, hence is equal to the sum of the dimension of $c$ with the dimension of $d$. 
\end{remark}

For a cube chain $c = (c_1,\ldots,c_l) \in Ch(X)^w_v$ of type $(n_1,\ldots,n_l)$ and dimension $i=\sum\limits_{j=1}^l n_i-l$, an integer $k \in \{1,\ldots,l\}$ and a subset $\mathcal{I} \subseteq \{0,\ldots,n_k-1\}$ having $r$ elements, where $0 < r < n_k$, define a cube chain (as done in \cite{cubechains}):
$$
d_{k,\mathcal{I}}(c) = (c_1,\ldots,c_{k-1},d^0_{\overline{\mathcal{I}}}(c_k),d^1_\mathcal{I}(c_k),c_{k+1},\ldots,c_l) \in Ch(X)^w_v$$ 
\noindent where $\overline{\mathcal{I}} = \{0,\ldots,n_k-1\} \backslash \mathcal{I}$. % and $d^k_B$ is defined as follows: 
This defines a cell complex with cells of dimension $i$, $Ch^{=i}(X)$, being $i$-cube chains, and with faces given by the $d_{k,\mathcal{I}}$.

For $R$ a given commutative field we define: 

\begin{definition}
\label{def:modcub}
Let $X$ be a precubical set. We call {${R_{i+1}[X]}$} for $i\geq 0$ the $R$-module generated by all cube chains of dimension $i$. %, and $R_1[X]$ be the free $R$-vector space generated by all cube chains of dimension 0 plus the vertices of $X$.    
\end{definition}

The formula below is given in %defined in 
\cite{cubechains}: 

\begin{definition}
\label{def:boundaryoperator}
Define a \look{boundary map} from the %$R$-vector space 
$R_{i+1}[X]$ 
%generated by $i$-cube chains of $X$ to the $R$-vector space 
to $R_i[X]$ %generated by $(i-1)$-cube chains of $X$  %as a linear combinations of such $d_{k,A}$ above: 
as follows: 
%We refer the reader to \cite{Ziemianski} for more details. 

\begin{multline*}
\partial c = \sum\limits_{k=1}^l \sum\limits_{r=1}^{n_k-1} \sum\limits_{\mbox{\tiny $\mathcal{I}$} \subseteq\{0,\ldots,n_k-1\}: \ \mid \mbox{\tiny $\mathcal{I}$} \mid=r}
(-1)^{n_1+\ldots+n_{k-1}+k+r+1} sgn(\mbox{$\mathcal{I}$}) {d_{k,\mbox{$\mathcal{I}$}}(c)}
\end{multline*}
\noindent where
$$
sgn(\mathcal{I})=\left\{\begin{array}{ll}
1 & \mbox{if $\sum\limits_{i\in \mathcal{I}} i \equiv \sum\limits_{i=1}^r i \ mod \ 2$} \\
-1 & \mbox{otherwise}
\end{array}\right.
$$
%\todo[inline]{A ecrire plus en detail}
Then $R_*[X]=(R_{i+1}[X],\partial)_{i\in \N}$ is a chain complex. The restriction to cube chains from any $v \in X_0$ to any $w \in X_0$ is a sub-chain complex of $R_*[X]$ that we write $R^w_{{v},*}[X]$.
\end{definition}

\section{Coboundary maps}
\label{sec:coboundarymaps}

The boundary maps $\partial$ that allow for defining a directed homology theory in \cite{directedpersistencemodule} are not only $R$-module maps, but also $R[X]$-bimodule maps, where $R[X]$ is the path algebra on the underlying quiver of $X$ (see \cite{directedpersistencemodule}). 

Now consider the following construction. Let $R^*_n[X]$ be the $R$-module $$Hom_{{ }_R Mod}(R_n[X],R)$$ 
\noindent which is the direct sum of $R$-modules $Hom_{{ }_R Mod}(R_n[X]^\beta_\alpha,R)$ for all $(\alpha,\beta)\in R_0[X]$, where $R_0[X]$ denotes the set of pairs of reachable states here.

We now define the coboundary map $\partial^*: \ R_i^*[X] \rightarrow R_{i+1}^*[X]$ to be, for $f \in R_i^*[X]$:
$$
\partial^*(f) = f \circ \partial
$$
\noindent which is an element of $R_{i+1}^*[X]$ since $\partial: \ R_{i+1}[X] \rightarrow R_i[X]$.

Consider $R[X]^{op}$ the opposite algebra of $R[X]$, i.e. the algebra which has the same underlying $R$-module as $R[X]$ but has as internal multiplication $\times^*$ the opposite internal multiplication $\times$ of $R[X]$: for $p$ and $q$ in $R[X]^{op}$, $p\times^* q=q\times p$. Now, define a $R[X]^{op}$-bimodule structure on $R_n^*[X]$ by, for $p$, $q$ in $R[X]^{op}$ and $f \in R_n^*[X]$, $x \in R_n[X]$: 
$$
p\bullet f \bullet q (x)=f(p\bullet f(x) \bullet q) 
$$
Indeed we have: 
$$\begin{array}{rcl}
p'\bullet (p\bullet f \bullet q)\bullet q'(x) & = & p\bullet f(p'\bullet x \bullet q') \bullet q\\
& = & f(p\bullet (p'\bullet x \bullet q')\bullet q) \\
& = & f((p\times p')\bullet x \bullet (q'\times q) \\
& = & f((p'\times^* p)\bullet x \bullet (q\times^* q')) \\
& = & (p'\times^* p) \bullet f \bullet (q\times^* q') (x)
\end{array}
$$

Indeed, the coboundary map we defined is a morphism of $R[X]^{op}$-bimodules: 

$$
\begin{array}{rcl}
\partial^*(p\bullet f \bullet q)(x) & = & (p\bullet f \bullet q)(\partial(x)) \\
& = & f(p\bullet \partial(x) \bullet q) \\
& = & f(\partial(p\bullet x \bullet q)) \\
& = & p\bullet (f \circ \partial) \bullet q(x) \\
& = & p \bullet \partial^*(f) \bullet q (x)
\end{array}
$$

\begin{remark}
For $\rho \in {R^*_n[X]}_\alpha^\beta$ (which is $Hom_{{ }_X Mod}({R_n[X]}_\alpha^\beta,R)$), and $p \in {R[X]^{op}}_\alpha^{\alpha'}$, $q \in {R[X]^{op}}_{\beta'}^\beta$, $p\bullet \rho \bullet q \in {R^*_n[X]}_{\alpha'}^{\beta'}$.
\end{remark}

\section{Directed Cohomology bimodules of certain precubical sets}

\label{sec:precubcohomology}

We can now define the cohomology $R[X]^{op}$-bimodule of a precubical set:

\begin{definition}
\label{def:cohomology}
The \look{cohomology modules} of a finite precubical set $X$ is defined as the cohomology $(HM^{i+1})_{i\geq 0}[X]$ in the abelian category of $R[X]^{op}$-bimodules, of the complex of $R[X]^{op}$-bimodules %\todo{Be careful with $HM^1[X]$ (similar to homology)} 
given by the coboundary operators $\partial^*: \ R_i^*[X]\rightarrow R_{i+1}^*[X]$: 
$$
HM^{i+1}[X]=Ker \ \partial^*_{\mid R^*_{i+1}[X]}/Im \ \partial^*_{\mid R^*_{i}[X]}
$$
\end{definition}

\begin{remark}
Note that for $i=0$, $Im \ \partial^*_{| R^*_i[X]}=0$ in the definition above. 
\end{remark}

By Lemma 5.7 of \cite{directedpersistencemodule}, 
the quotient 
$HM^{i+1}[X]=Ker \ \partial^*_{\mid R^*_{i+1}[X]}/Im \ \partial^*_{\mid R^*_{i}[X]}$ can be identified, as an $R$-module, with the coproduct of all $R$-modules  
$$Ker \ \partial^*_{R^*_{i+1}[X]_a^b}/Im \ \partial^*_{R^*_{i}[X]_a^b}$$ 
\noindent that we denote by $HM^{i+1}[X]^b_a$. 
These cohomology modules characterize the cohomology of the trace spaces, as shown below: 

\begin{lemma}
\label{lem:homtrace}
Let $X$ be a finite precubical set
that is such that is has proper non-looping length covering, $a, b \in X_0$. Then
the $R$-module $HM^n[X]^b_a$, $n \geq 1$, is the standard $(n-1)$th cohomology of the trace space $Tr(| X |)^b_a$ from $a$ to $b$. 
\end{lemma}

\begin{proof}
$HM^*[X]_a^b$ as defined above, is the cohomology of the cochain complex $Ch(X)^b_a$, and by Theorem 1.7 of \cite{cubechains}, this is isomorphic to the (singular) cohomology of the trace space $Tr(\mid X \mid)^b_a$.
\end{proof}

\section{Directed Homology operations}

\label{sec:precuboperations}

There is a tensor product between $R_i[X]$ and $R_j[X]$ as follows. For $c=(c_1,\ldots,c_n)$ an $(i+1)$-cube chain of $X$ from $\alpha$ to $\beta$, $d=(d_1,\ldots,d_m)$ an $(j+1)$-cube chain from $\beta$ to $\gamma$, we define: 

$$
c\otimes d=(c_1,\ldots,c_n,d_1,\ldots,d_m)
$$

When cube chains $c$ and $d$ do not have matching end points, we define $c\otimes d$ to be equal to 0. Then $\otimes$ is extended to $R_i[X]\times R_j[X]$ by bilinearity. We note that $\otimes$ then maps $R_i[X]\times R_j[X]$ to $R_{i+j-1}[X]$. The tensor product is just concatenation of higher cube paths, extending the concatenation operation $\bullet$. 

We note now the following:

\begin{lemma}
\label{lem:diffalg}
    The boundary operator $\partial$ has the property that, for any $c$ $(i+1)$-cube path and $d$ any $(j+1)$-cube path:
$$
\partial(c\otimes d) = \partial(c)\otimes d +(-1)^{|c|} c\otimes \partial(d)
$$
\end{lemma}

\begin{proof}
Suppose first that $c$ and $d$ are cube chains

When $c$ and $d$ do not have matching ends, the equality above is trivial (everything is 0).

Suppose now that $c$ and $d$ have matching ends. 
Let $c=(c_1,\ldots,c_n)$ (so that $|c|=n$ is the length of $c$), and $d=(d_1,\ldots,d_m)$. Therefore we have $c\otimes d=(c_1,\ldots,c_n,d_1,\ldots,d_m)$.

Now: 

$$
\begin{array}{rcl}
{\scriptstyle \partial (c\otimes d)} & {\scriptstyle =} &\sum\limits_{k=1}^{n+m} \sum\limits_{r=1}^{n_k-1} \sum\limits_{\scriptscriptstyle \mathcal{I} \subseteq\{0,\ldots,n_k-1\}: \ \mid \mathcal{I} \mid=r}
{\scriptstyle (-1)^{\scriptscriptstyle n_1+\ldots+n_{k-1}+k+r+1} sgn(\mbox{$\scriptstyle \mathcal{I}$}) d_{k,\mbox{$\scriptscriptstyle \mathcal{I}$}}(c\otimes d)} \\
& {\scriptstyle =} & \sum\limits_{k=1}^{n} \sum\limits_{r=1}^{n_k-1} \sum\limits_{\scriptscriptstyle \mathcal{I} \subseteq\{0,\ldots,n_k-1\}: \ \mid \mathcal{I} \mid=r}
{\scriptstyle (-1)^{\scriptscriptstyle n_1+\ldots+n_{k-1}+k+r+1} sgn(\mbox{$\scriptstyle \mathcal{I}$}) d_{k,\mbox{$\scriptscriptstyle \mathcal{I}$}}(c\otimes d)} \\
& & +\sum\limits_{k=n+1}^{m} \sum\limits_{r=1}^{n_k-1} \sum\limits_{\scriptscriptstyle \mathcal{I} \subseteq\{0,\ldots,n_k-1\}: \ \mid \mathcal{I} \mid=r} 
{\scriptstyle (-1)^{\scriptscriptstyle n_1+\ldots+n_{k-1}+k+r+1}
sgn(\mbox{$\scriptstyle \mathcal{I}$}) d_{k,\mbox{$\scriptscriptstyle \mathcal{I}$}}
 (c\otimes d)} \\ 
& {\scriptstyle =} & \sum\limits_{k=1}^{n} \sum\limits_{r=1}^{n_k-1} \sum\limits_{\scriptscriptstyle \mathcal{I} \subseteq\{0,\ldots,n_k-1\}: \ \mid \mathcal{I} \mid=r}
{\scriptstyle (-1)^{\scriptscriptstyle n_1+\ldots+n_{k-1}+k+r+1} sgn(\mbox{$\scriptstyle \mathcal{I}$}) d_{k,\mbox{$\scriptscriptstyle \mathcal{I}$}}(c)\otimes d} \\
& & +(-1)^n \sum\limits_{k=1}^{m} \sum\limits_{r=1}^{n_k-1} \sum\limits_{\scriptscriptstyle \mathcal{I} \subseteq\{0,\ldots,n_k-1\}: \ \mid \mathcal{I} \mid=r}
{\scriptstyle (-1)^{\scriptscriptstyle n_1+\ldots+n_{k-1}+k+r+1} sgn(\mbox{$\scriptstyle \mathcal{I}$}) c\otimes d_{k,\mbox{$\scriptscriptstyle \mathcal{I}$}}(d)} \\
& {\scriptstyle =} & {\scriptstyle \partial(c)\otimes d +(-1)^{|c|}c\otimes d}
\end{array}
$$
%
%The case when $c$ and $d$ are general elements of $R_i[X]$ and of $R_j[X]$ respectively stems out from bilinearity of the tensor operation. 
\end{proof}

This tensor operation generates in turn maps in homology as follows, at least for a restricted case of precubical set. 

For the rest of the paper, {\bf we now assume that all $m$-cube-chains $c$ in $Ch^{=m}(X)_v^w$, for all $v, w \in X_0$ have the same length $|c|$ (i.e. the length of cube chains only depends on its dimension and beginning and end vertices)}. This is in particular the case for pre-cubical complexes which give the semantics of non-looping PV-programs \cite{thebook}. 

In that case, for any $c\in Ch^{=i+1}(X)_{\alpha}^{\beta}$ and $d\in Ch^{=j+1}(X)_{\beta}^{\gamma}$, $\partial(c\otimes d) = \partial(c)\otimes d +(-1)^{|c|} c\otimes \partial(d)$ as well, because we can extend by bilinearity the result of Lemma \ref{lem:diffalg}: by hypothesis, all $(i+1)$-cube-chains from $\alpha$ to $\beta$ have the same length $|c|$. Now:

\begin{lemma}
\label{lem:conc}
We have homological operations called conc-products defined by maps 
$$\circledast: \ HM_i[X]_\alpha^\beta \otimes HM_j[X]_\beta^\gamma \rightarrow HM_{i+j-1}[X]_\alpha^\gamma$$ 
\noindent for all $i\geq 1$, $j\geq 1$ and $(\alpha,\beta) \in R_0[X]$. 
\end{lemma}

\begin{proof}
The proof goes as with any boundary operator of a graded differential module. First, we note that products of cycles are cycles, since, if $\partial(c)=\partial(d)=0$,  
$$
\begin{array}{rcl}
\partial(c\otimes d) & = & \partial(c)\otimes d +(-1)^{|c|} c\otimes \partial(d) \\
& = & 0
\end{array}
$$

Now, we note that the product of a cycle with a boundary is a boundary, as, if $c$ is such that $\partial(c)=0$ and $d=\partial(D)$, 
$$
\begin{array}{rcl}
\partial(c\otimes D) & = & \partial(c)\otimes D +(-1)^{|c|} c\otimes \partial(D) \\
& = & (-1)^{|c|} c \otimes d
\end{array}
$$
Therefore $c \otimes d=(-1)^{|c|} \partial(c\otimes D)\in Im \ \partial$ is a boundary. 

We define the $\circledast$ operation on classes $[c]$ of a cycle $c$ in $HM_i[X]^\beta_\alpha$ and $[d]$ of a cycle $d$ in $HM_j[X]^\gamma_\beta$ to be equal to $[c]\circledast [d]=[c\otimes d]$. We now check that this definition is well-formed. 

Indeed, suppose that $[c']=[c]$ and $[d']=[d]$, that is, there exists $C \in R_{i+1}[X]_\alpha^\beta$ and $D \in R_{j+1}[X]_\beta^\gamma$ such that $c'=c+\partial(C)$ and $d'=d+\partial(D)$. Then: 

$$
\begin{array}{rcl}
c'\otimes d' & = & c\otimes d+c\otimes \partial(D)+\partial(C)\otimes d + \partial(C)\otimes \partial(D)
\end{array}
$$
But $c$ is a cycle and $\partial(D)$ is a boundary, so $c\otimes \partial(D)$ is a boundary as we just saw above. Similarly, $\partial(C)\otimes d$ is a coboundary. Now, $\partial(D)$ is a boundary, so is a cycle, in particular, since $\partial^2=0$, hence $\partial(C)\otimes \partial(D)$ is the product of a boundary with a cycle, so is a boundary as well. Overall, we conclude that $c'\otimes d'$ is equal to $c\otimes d$ modulo a boundary, hence $[c'\otimes d']=[c\otimes d]$ making the definition of $\circledast$ well posed. 
\end{proof}

\begin{remark}
Note that for $i$ or $j$ equal to 1, conc-products are maps from $HM_1[X]_\alpha^\beta \times HM_j[X]_\beta^\gamma$ to $HM_1[X]_\alpha^\gamma$. Also, this implies that the module $HM_1[X]$ is actually an algebra using the internal multiplication $\circledast$. Overall, we can replace the bimodule homology by a reduced version, where all $HM_i[X]$ are $HM_1[X]$-bimodules, $HM_1[X]$ being also an algebra. 
\end{remark}

\begin{remark}
We claim Lemma \ref{lem:conc} still holds for all finite precubical sets with proper non-looping length covering. A proof will need to consider the construction above on the non-looping length covering of $X$ and then use  Proposition 5.3 of \cite{RaussenII}. 
\end{remark}

\section{Directed Cohomology operations}

\label{sec:directedoperationscub}

\subsection{Local cup-product}

\begin{lemma}
There is a $R$-module map 
$$\smile: \ HM^i[X]_\alpha^\beta \times HM^j[X]_\alpha^\beta \rightarrow HM^{i+j-1}[X]_\alpha^\beta$$ induced by the cup-product of the trace spaces of $|X|$ from $\alpha$ to $\beta$ in $X_0$. 
\end{lemma}

\begin{proof}
We know by Lemma \ref{lem:homtrace} that 
the $R$-module $HM^n[X]^b_a$, $n \geq 1$, and $a, b \in X_0$, is the standard $(n-1)$th cohomology of the trace space $T(| X |)^b_a$ from $a$ to $b$. We indeed have a cup-product on the cohomology of $Tr(| X |)^b_a$, hence 
all $R$-modules $HM^*[X]_\alpha^\beta$ are equipped with a cup-product $\smile: \ HM^i[X]_\alpha^\beta \times HM^j[X]_\alpha^\beta \rightarrow HM^{i+j-1}[X]_\alpha^\beta$. 
\end{proof}

The cochain complex on which $HM^*[X]$ is computed is given by a CW-complex, in theory, it is feasible to give an explicit formula for the cup-product, using the formulas of \cite{Whitney}, or by going through a simplicial version of the underlying CW-complex. This is left for future work.

\begin{remark}
Consider $p$ a 0-cube chain from $\alpha$ to $\alpha'$ in $X_0$ and $q$ a 0-cube chain from $\beta'$ to $\beta$. The action of $p$ on the left (and respectively $q$ on the right), on element $x\in HM^i[X]_\alpha^\beta$ (and on respectively $y\in HM^j[X]$) gives elements
$p\bullet x \bullet q \in HM^i[X]$ and $p\bullet y \bullet q \in HM^j[X]$. We claim is that $p\bullet (x \smile y) \bullet q=(p\bullet x\bullet q)\smile (p\bullet y\bullet q)$, this is left for future work. 
\end{remark}

\subsection{Cohomological inner-tensor}

In this section, we dualize the homological operation $\circledast$ by defining the following: 

\begin{definition}
We define a dual tensor operation on cochains $f\in R^*_i[X]_\alpha^\beta$ and $g \in R^*_j[X]_\beta^\gamma$ by, for $c \in R_i[X]_\alpha^\beta$ and $d\in R_j[X]_\beta^\gamma$:
$$
f\boxtimes g(c\otimes d) = f(c).g(d)
$$
\noindent (where $.$ denotes the product in $R$ and where $f\boxtimes g(z)=0$ for $z$ $(i+j-1)$-cube chain from $\alpha$ top $\gamma$ that does not go through $\beta$)
\end{definition}

Overall the $\boxtimes$ operation gives a map from $\coprod_{\scriptscriptstyle \beta \mbox{ \tiny s.t. } (\alpha,\beta)\in R_0[X], (\beta,\gamma)\in R_0[X]} {R^*_i[X]}_\alpha^\beta \times {R^*_j[X]}_\beta^\gamma$ to ${R^*_{i+j-1}[X]}_\alpha^\beta$. 

Now, consider $Y \in {R_{i+j-1}[X]}_\alpha^\gamma$ that are tensor products $Y=Y_1\otimes Y_2$ with $Y_1 \in {R_i[X]}_\alpha^\beta$, $Y_2\in {R_j[X]}_\beta^\gamma$, for some $\beta$, then: 

$$
\begin{array}{rcl}
\partial^*(f\boxtimes g)(Y) & = & f\boxtimes g(\partial(Y)) \\
& = & f\boxtimes g(\partial(Y_1)\otimes Y_2+(-1)^{|Y_1|}Y_1\otimes \partial(Y_2)) \\
& = & f\boxtimes g(\partial(Y_1)\otimes Y_2)+(-1)^{|Y_1|}f\boxtimes g(Y_1\otimes \partial(Y_2)) \\
& = & f(\partial(Y_1)).g(Y_2)+(-1)^{|Y_1|}f(Y_1).g(\partial(Y_2)) \\
& = & \partial^*(f)(Y_1).g(Y_2)+(-1)^{|Y_1|}f(Y_1).\partial^*(g)(Y_2) \\
%& = & f(\partial(Y_1)).g(Y_2)+(-1)^{|Y_1|}f(Y_1).g(\partial(Y_2) \\
& = & \partial^*(f)\boxtimes g(Y_1\otimes Y_2)+(-1)^{|Y_1|}f\boxtimes \partial^*(g)(Y_1\otimes Y_2)
\end{array}$$
\noindent (since we are still in the restricted case of precubical sets that are such that $m$-cube chains from $\alpha$ to $\beta$ have the same length). 

\begin{lemma}
The product $\boxtimes$ induces a $R$-module map 
$\curvearrowright$ from ${HM^i[X]}_\alpha^\beta \times {HM^j[X]}_\beta^\gamma$ to ${HM^{i+j-1}[X]}_\alpha^\beta$.
\end{lemma}

\begin{proof} 
Similarly as with the proof of Lemma \ref{lem:conc}, this implies that the $\boxtimes$ product of a cocycle in $R^*_{i}[X]^\beta_\alpha$ with a cocycle in $R^*_{j}[X]^\beta_\alpha$ is a cocycle in $R^*_{i+j-1}[X]^\beta_\alpha$, and the product of a cocycle with a coboundary is a coboundary, showing that 
$\boxtimes$ goes over to cohomological classes. 
\end{proof}

%\coprod_{\scriptscriptstyle \beta \mbox{ \tiny s.t. } (\alpha,\beta)\in R_0[X], (\beta,\gamma)\in R_0[X]}

\section{Homological and cohomological operations in the case of general $d$-spaces}

\label{sec:dspace}

We now recap, for completeness, the construction of Section 6 of \cite{semi-ab}. 
We recall the following: 
%\begin{definition}
%\label{def:standardsimplex}
the standard simplex of dimension $n$ is $$
\Delta_n=\left\{(t_0,\ldots,t_n) \mid \forall i\in \{0,\ldots,n\}, \ t_i \geq 0 \mbox{ and } \sum\limits_{j=0}^n t_j=1\right\}
$$
For $n \in \N$, $n
\geq 1$ and $0 \leq k \leq n$, the $k$th ($n-1$)-face (inclusion) of the topological $n$-simplex is the subspace inclusion
$$\delta_k: \ \Delta_{n-1} \rightarrow \Delta_n$$
induced by the inclusion
$$(t_0,\ldots,t_{n-1}) \rightarrow (t_0,\ldots,t_{k-1},0,t_k,
\ldots, t_{n-1})$$
For $n \in \N$ and $0\leq k < n$, the 
$k$th degenerate 
$n$-simplex is the surjective map
$$
\sigma_k: \ \Delta_n \rightarrow \Delta_{n-1}$$
\noindent induced by the surjection: 
$$(t_0,\ldots,t_n)\rightarrow (t_0,\ldots,t_{k}+t_{k+1},\ldots,t_n)$$
%\end{definition}

Still as in \cite{semi-ab},  
we call $p$ an $i$-trace, $i\geq 1$, or trace of dimension $i$ of a directed space $\Vec{X}$, any 
 continuous map 
$$p: \Delta_{i-1} \rightarrow Tr(\Vec{X})$$ 
\noindent which is such that: 
\begin{itemize}
%\item modulo continuous increasing reparameterization in the first coordinate, %and continuous reparameterizations in the other coordinates 
\item $p(t_0,\ldots,t_{i-1})(0)$ does not depend on $t_0,\ldots,t_{i-1}$ and we denote it by $s_p$ (``start of $p$")
\item and $p(t_0,\ldots,t_{i-1})(1)$ is constant as well, that we write as $t_p$ (``target of $p$")
\end{itemize}
We write $T_i(\Vec{X})$ for the set of $i$-traces in $\Vec{X}$. We write $T_i(\Vec{X})(a,b)$, $a$, $b \in \Vec{X}$, for the subset of $T_i(\Vec{X})$ made of $i$-traces from $a$ to $b$. 

In short, an $i$-trace of $\Vec{X}$ is a particular $(i-1)$-simplex of the trace space of $\Vec{X}$, for which we can define boundary and degeneracy maps, as usual: 
%\todo[inline]{To be continued - curryfier partout...}

\begin{itemize}
\item Maps $d_{j}$, $j=0,\ldots,i$ acting on $(i+1)$-paths %$T_{i+1}(\Vec{X})$ ($i\geq 1$) by, for 
$p: \ \Delta_i \rightarrow Tr(\Vec{X})$, $i\geq 1$: %an $(i+1)$-trace: 
$$d_j(p)=p\circ \delta_j$$
%\noindent with $\delta_j: \Delta_{i-1} \rightarrow \Delta_i$ is the inclusion map of Definition \ref{def:standardsimplex}
%(s_1,\ldots,s_{i+1})=p(s_1,s_2,\ldots,0_{j+2},\ldots,s_{i})$$
\item Maps $s_k$, $k=0,\ldots,i-1$ acting on $i$-paths 
%$T_i(\Vec{X})$ by, for 
$p: \Delta_{i-1} \rightarrow Tr(\Vec{X})$, $i\geq 1$: % an $i$-trace:
$$s_k(p)=p\circ \sigma_k$$
%\noindent where $\sigma_k: \Delta_{i+1}\rightarrow \Delta_{i}$ is the surjective map of Definition \ref{def:standardsimplex}
\end{itemize}

The concatenation operation on dipaths, that we denote by $*$ obviously acts on the left and on the right on $Tr_i(\Vec{X})_\alpha^\beta$ when ends meet. These operations commute with the boundary and degeneracy operators, hence 
we can then define, as for precubical sets: 

\begin{definition}
% Homology
The directed homology bimodule of the directed space $\Vec{X}$ is $HM_i[\Vec{X}]=Ker \ \partial_{|C^i[\Vec{X}]}/Im \ \partial_{|C^{i+1}}[\Vec{X}]$.

The directed cohomology bimodule of the directed space $\Vec{X}$ is $HM^i[\Vec{X}]=Ker \ \partial^*_{|C^i[\Vec{X}]}/Im \ \partial^*_{|C^{i-1}}[\Vec{X}]$.
\end{definition}

\begin{remark}
Note that formally, things are a bit more difficult to define here, since the path algebra $R[\Vec{X}]$ is not unital for general $d$-spaces. We refer the reader again to \cite{eliot} and \cite{albert} for more details on how to conceive modules over non-unital algebras, in the particular context of directed homology. 
\end{remark}

Similarly to the case of precubical sets in Section \ref{sec:precuboperations},  concatenation $*$ induces homological and cohomological operations. 

Indeed, the concatenation operation on traces is a continuous maps from $Tr(X)^\beta_\alpha \times Tr(X)_\beta^\gamma$ to $Tr(X)_\alpha^\gamma$, therefore induces maps from $H_{i+j-2}(Tr(X)_\alpha^\beta
\times Tr(X)_\beta^\gamma)$
 to 
$H_{i+j-2}(Tr(X)_\alpha^\gamma$. For simplicity's sake, {\bf we are now supposing $R$ is a field}, then, by K\"unneth formula, $H_{i+j-2}(Tr(X)_\alpha^\beta
\times Tr(X)_\beta^\gamma)$ is isomorphic to: 
$$\mathop{\bigoplus}\limits_{k+l=i+j-2}
H_k(Tr(X)_\alpha^\beta)\times H_l(Tr(X)_\beta^\gamma)=\mathop{\bigoplus}\limits_{k+l=i+j}
HM_k[X]_\alpha^\beta\times HM_l[X]_\beta^\gamma$$
 to $HM_{i+j-1}[M]_\alpha^\gamma$. 

Similarly, there is a continuous map ${ }^\beta Tr(X)^\gamma_\alpha$, the subspace of traces in $X$ from $\alpha$ to $\gamma$, going through $\beta$, to $Tr(X)_\alpha^\beta\times Tr(X)_\beta^\gamma$. Using again K\"unneth formula, in cohomology, and the contravariance of the cohomology functor, we get a map from 
$\mathop{\bigoplus}\limits_{k+l=i+j}
HM^k[X]_\alpha^\beta\times HM^l[X]_\beta^\gamma$ to 
$H^{i+j-2}({ }^\beta Tr(X)_\alpha^\gamma)$. Finally, we note that we have a $R$-module map from $H^{i+j-2}({ }^\beta Tr(X)_\alpha^\gamma)$ to $H^{i+j-2}(Tr(X)_\alpha^\gamma)$ by mapping the class of a cochain $c$ in $H^{i+j-2}({ }^\beta Tr(X)_\alpha^\gamma)$ to the class of a cochain $d$ in $H^{i+j-2}(Tr(X)_\alpha^\gamma)$ by mapping $(i+j-2)$-traces $p$ of $X$ from $\alpha$ to $\gamma$ to $c(p)$ if $p$ goes through $\beta$, otherwise to $0$. Hence we have: 

\begin{definition}
% Cup-product
We call $$\circledast: \ \mathop{\bigoplus}\limits_{i+j=n} HM_i[\Vec{X}]_\alpha^\beta \otimes HM_j[\Vec{X}]_\beta^\gamma \rightarrow HM_{n-1}[\Vec{X}]_\alpha^\gamma$$ 
\noindent for all %$i\geq 1$, $j\geq 1$ and 
$\alpha,\beta,\gamma \in \Vec{X}$ such that $\beta$ is reachable from $\alpha$, and $\gamma$ is reachable from $\beta$.

We call $$\curvearrowright: \ \bigoplus\limits_{i+j=n} HM^i[\Vec{X}]_\alpha^\beta \otimes HM^j[\Vec{X}]_\beta^\gamma \rightarrow HM^{n-1}[\Vec{X}]_\alpha^\gamma$$ the induced operation in cohomology. 
\end{definition}

Finally, we have the classical ``local" cup products: 

\begin{definition}
We denote by: 
$$\smile: \ HM^i[\Vec{X}]_\alpha^\beta \times HM^j[\Vec{X}]_\alpha^\beta \rightarrow HM^{i+j-1}[\Vec{X}]_\alpha^\beta$$
\noindent the cup-product between the $(i-1)$th cohomology class and the $(j-1)$th cohomology class of the trace space of $\Vec{X}$ between points $\alpha$ and $\beta$ in $\Vec{X}$, such that $\beta$ is reachable from $\alpha$ in $\Vec{X}$. 
\end{definition}

\section{Examples of cohomological calculations}

\label{sec:examples}

Consider $n$ processes in the PV language of \cite{thebook}, in parallel, that can only synchronize weakly using $(n-1)$-semaphores.

As far as the cohomology of the trace space between two pair of points (identified with tuples of local times at which some action is started, or ended) is involved, it is easy to see that an equivalent geometric model to the classical cubical semantics (a subdivided hypercube of dimension $n$ where some $n$ cubes are missing) of such semaphore based programs \cite{book} is given by $[0,k_1]\times [0,k_2]\times \ldots [0,k_n] \backslash \mathcal{O} \subseteq \R^n$, where $\mathcal{O}$ is a set of points $\{\mathcal{O}_1,\ldots,\mathcal{O}_k\}$ with half integer coordinates. 

We further constrain the possible synchronizations to have that: 
$$
\forall l=1,\ldots,d, \ \forall i\neq j, \ 1 \leq i, j \leq k, \ \pi_l(\mathcal{O}_i)\neq \pi_l(\mathcal{O}_j) 
$$

In this section, we take $R=\Z$. 

\begin{example}
\label{ex:ex2D}
We exemplify this below. On the left is the precubical set of dimension 2 (2-cells are indicated by the diagonal double lines) and on the right, the ``equivalent" $d$-space, which is a square without four points $\mathcal{O}_1$, $\mathcal{O}_2$, $\mathcal{O}_3$, $\mathcal{O}_4$, delineated. It is obtained by ``shrinking" the empty 2-cells to a point. By equivalent, we mean that there is a bisimulation equivalence of their corresponding natural homologies, in the sense of \cite{Dubut}. 

We also superpose on the right figure below  the skeleton of the precubical set on the left, so that to obtain relevant pair of points at which compute the cohomology classes of the corresponding trace space: 

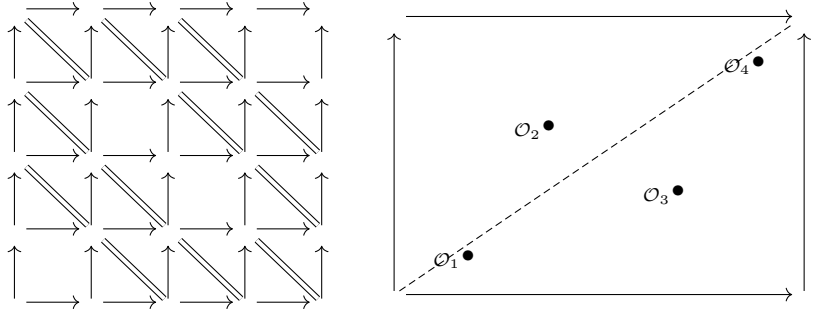
\begin{center}
\begin{minipage}{.45\textwidth}
\[
\begin{tikzcd}[column sep=1em, row sep=1em]
\arrow{rr}\ar[rrdd,equal] & & \arrow{rr}\ar[rrdd,equal] & & \arrow{rr}\ar[rrdd,equal] & & \arrow{rr} & &\mbox{} \\
& & & & & & & & \\
\arrow{rr}\arrow{uu}\ar[rrdd,equal] & & \arrow{rr}\arrow{uu} & & \arrow{rr}\arrow{uu}\ar[rrdd,equal] & & \arrow{rr}\arrow{uu}\ar[rrdd,equal] & & \arrow{uu} \\
& & & & & & & & \\
\arrow{rr}\arrow{uu}\ar[rrdd,equal] & & \arrow{rr}\arrow{uu}\ar[rrdd,equal] & & \arrow{rr}\arrow{uu} & & \arrow{rr}\arrow{uu}\ar[rrdd,equal] & & \arrow{uu} \\
& & & & & & & & \\
\arrow{rr}\arrow{uu} & & \arrow{rr}\arrow{uu}\ar[rrdd,equal] & & \arrow{rr}\arrow{uu}\ar[rrdd,equal] & & \arrow{rr}\arrow{uu}\ar[rrdd,equal] & & \arrow{uu} \\
& & & & & & & & \\
\arrow{rr}\arrow{uu} & & \arrow{rr}\arrow{uu} & & \arrow{rr}\arrow{uu} & & \arrow{rr}\arrow{uu} & & \arrow{uu}
\end{tikzcd}
\]
\end{minipage}
\begin{minipage}{.45\textwidth}
\[
\begin{tikzcd}[column sep=0.6em, row sep=0.5em]
\arrow{rrrrrrrr} & &  & &  & &  & &\mbox{} \\
& & & & & & & { }_{\mathcal{O}_4} \bullet & \\
& & & & & & & &  \\
& & { }_{\mathcal{O}_2} \bullet & & & & & & \\
& & & & & & & & \\
& & & & & & { }_{\mathcal{O}_3} \bullet & & \\
& & & & & & & & \\
& { }_{\mathcal{O}_1} \bullet & & & & & & & \\
%& & & & & & & & \\
\arrow{rrrrrrrr}\arrow{uuuuuuuu}\arrow[dashed,dash]{urururururururur} & &  & &  & &  & & \arrow{uuuuuuuu}
\end{tikzcd}
\]
\end{minipage}
\end{center}

Using as global time the diagonal of the square above, we see that we are in the case of \cite{Barychnikov}, with the same future cone as Example 3.1.1 of \cite{Barychnikov}. 

From \cite{Barychnikov} again, we know that the cohomology of the trace space of that $d$-space is concentrated in degree 0. Let us compute the generators, for all pairs of points in the skeleton of the precubical set
indicated below: 

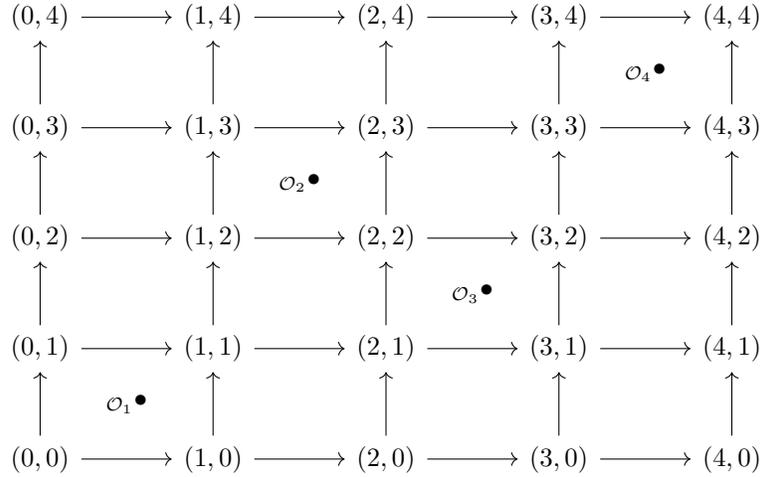
\[
\begin{tikzcd}[column sep=0.5em, row sep=0.5em]
(0,4) \arrow{rr} & & (1,4) \arrow{rr} & & (2,4)\arrow{rr} & & (3,4) \arrow{rr} & & (4,4) \mbox{} \\
& & & & & & & { }_{\mathcal{O}_4} \bullet & \\
(0,3) \arrow{rr}\arrow{uu} & & (1,3) \arrow{rr}\arrow{uu} & & (2,3) \arrow{rr}\arrow{uu} & & (3,3) \arrow{rr}\arrow{uu} & & (4,3) \arrow{uu} \\
& & & { }_{\mathcal{O}_2} \bullet & & & & & \\
(0,2) \arrow{rr}\arrow{uu} & & (1,2) \arrow{rr}\arrow{uu} & & (2,2) \arrow{rr}\arrow{uu} & & (3,2) \arrow{rr}\arrow{uu} & & (4,2) \arrow{uu} \\
& & & & &  { }_{\mathcal{O}_3} \bullet & & & \\
(0,1) \arrow{rr}\arrow{uu} & & (1,1) \arrow{rr}\arrow{uu} & & (2,1) \arrow{rr}\arrow{uu} & & (3,1) \arrow{rr}\arrow{uu} & & (4,1) \arrow{uu} \\
& { }_{\mathcal{O}_1} \bullet & & & & & & & \\
(0,0)\arrow{rr}\arrow{uu} & & (1,0) \arrow{rr}\arrow{uu} & & (2,0)\arrow{rr}\arrow{uu} & & (3,0) \arrow{rr}\arrow{uu} & & (4,0)\arrow{uu}
\end{tikzcd}
\]

From \cite{Barychnikov}, we know that the 0th cohomology classes of the trace space between points $u$ and $v$ are in bijection with the chains of obstacle $\mathcal{O}_i$ starting with $u$ and ending in $v$. Note that the partial order is given componentwise and that as $\mathcal{O}_1=(1/2,1/2)$, $\mathcal{O}_2=(3/2,5/2)$, $\mathcal{O}_3=(5/2,3/2)$ and $\mathcal{O}_4=(7/2,7/2)$, we have $\mathcal{O}_1 < \mathcal{O}_2, \mathcal{O}_3 < \mathcal{O}_4$. 

For e.g. $u=(0,0)$ and $v=(5,5)$ we get the following 12 1th directed cohomology classes (corresponding to the 0th cohomology classes of the trace space between $u$ and $v$):
\begin{itemize}
    \item $(u < v)$
    \item $(u < \mathcal{O}_1 < v)$, $(u < \mathcal{O}_2 < v)$, $(u < \mathcal{O}_3 < v)$, $(u < \mathcal{O}_4 < v)$
    \item $(u < \mathcal{O}_1 < \mathcal{O}_2 < v)$, $(u < \mathcal{O}_1 < \mathcal{O}_3 < v)$, 
    $(u < \mathcal{O}_1 < \mathcal{O}_4 < v)$, 
    $(u < \mathcal{O}_2 < \mathcal{O}_4 < v)$, 
    $(u < \mathcal{O}_3 < \mathcal{O}_4 < v)$, 
    \item $(u < \mathcal{O}_1 < \mathcal{O}_2 < \mathcal{O}_4 < v)$, 
    $(u < \mathcal{O}_1 < \mathcal{O}_3 < \mathcal{O}_4 < v)$
\end{itemize}
Note that, from \cite{Barychnikov}, we know the cup product is idempotent: for all classes $c$, $c\smile c=c$. 

Now, the first directed cohomology classes between different pair of points $(u,v)$ are a subset of these classes, given by the same chains where they make sense. For instance, for $u=(0,0)$ and $v=(3,2)$, we get the 4 cohomology classes:
\begin{itemize}
    \item $(u < v)$
    \item $(u < \mathcal{O}_1 < v)$, $(u < \mathcal{O}_3 < v)$, 
    \item $(u < \mathcal{O}_1 < \mathcal{O}_3 < v)$
\end{itemize}

And for $u=(3,2)$, $v=(4,4)$, we get:
\begin{itemize}
\item $u < v$,
\item $u < \mathcal{O}_4 < v$
\end{itemize}

Let us exemplify the cohomological operation:
$$\curvearrowright: \ HM^1[X]^{(3,2)}_{(0,0)}\times HM^1[X]^{(4,4)}_{(3,2)}\rightarrow HM^1[X]^{(4,4)}_{(0,0)}$$ 
\noindent which concatenates these chains of obstacles, giving 8 0th cohomology classes:
\begin{itemize}
        \item $(u < v)$
    \item $(u < \mathcal{O}_1 < v)$, $(u < \mathcal{O}_3 < v)$, $(u < \mathcal{O}_4 < v)$
    \item $(u < \mathcal{O}_1 < \mathcal{O}_3 < v)$, 
    $(u < \mathcal{O}_1 < \mathcal{O}_4 < v)$, 
    $(u < \mathcal{O}_3 < \mathcal{O}_4 < v)$, 
    \item  
    $(u < \mathcal{O}_1 < \mathcal{O}_3 < \mathcal{O}_4 < v)$
\end{itemize}
\noindent which are the only chains of obstacles not containing $\mathcal{O}_2$ (the only obstacle not reachable between $(0,0)$ and $(3,2)$ as well as between $(3,2)$ and $(4,4)$). 
\end{example}

\begin{example}
\label{ex:3D}
Let us now consider the following example, in dimension 3: 
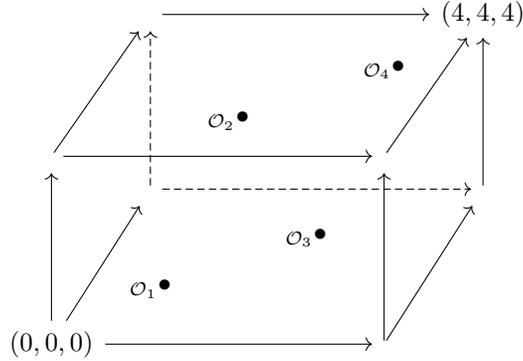
\[
\begin{tikzcd}[column sep=0.5em, row sep=0.5em]
& \arrow{rrrr} & & & & (4,4,4) \\
& & & & { }_{\mathcal{O}_4} \bullet & \\
& & { }_{\mathcal{O}_2} \bullet & & \\
 \arrow{uuur}\arrow{rrrr} & & & & \arrow{uuur} \\
%& & & & \\
& \arrow[dashed]{uuuu}\arrow[dashed]{rrrr} & & &  & \arrow{uuuu} & \\
& & & { }_{\mathcal{O}_3} \bullet & & \\
& { }_{\mathcal{O}_1} \bullet & & & & \\
(0,0,0) \arrow{uuuu}\arrow{uuur}\arrow{rrrr} & & & & \arrow{uuuu}\arrow{uuur}
\end{tikzcd}
\]
\noindent with the four points having coordinates $(1/2,1/2,1/2)$, $(3/2,5/2,3/2)$, $(5/2,3/2,$ $5/2)$ and $(7/2,7/2,7/2)$. 
Note that the accessibility relation between obstacles $\mathcal{O}_1$, $\mathcal{O}_2$, $\mathcal{O}_3$ and $\mathcal{O}_4$ is the same as for Example \ref{ex:ex2D}. 

Now, the obstacles $\mathcal{O}_i$ are of codimension 2, its corresponding avoidance class as defined in \cite{Barychnikov} is 
$c_i \in H^1(\R^2\backslash \mathcal{O}_i)$ (since $\R^2\backslash \mathcal{O}_i$ is homotopy equivalent to the 1-sphere $S^1$). 

Thus, the {cohomology ring of the trace space between $u=(0,0,0)$ and $v=(4,4,4)$}, hence the directed cohomology modules of the corresponding directed space 
are given by generators $c_1,\ldots, c_4$ with (Theorem 3.7 of \cite{Barychnikov}):
\begin{itemize}
\item In dimension 1 (0 in the trace space): 1 generator $\emptyset$ (corresponding to the empty chain of obstacles)
\item In dimension 2 (1 in the trace space): 4 generators $c_1$, $c_2$, $c_3$, $c_4$,
\item In dimension 3 (2 in the trace space): 5 generators $c_1 \smile c_2$, $c_1 \smile c_3$, $c_1\smile c_4$, $c_2\smile c_4$, $c_3\smile c_4$,
\item In dimension 4 (3 in the trace space): 2 generators $c_1 \smile c_2 \smile c_4$, $c_1 \smile c_3 \smile c_4$
\item All the other cohomology modules are trivial. 
\end{itemize}

If we change the end points, we get a similar calculation. For instance, for $u=(0,0,0)$ and $v=(2,3,2)$, 
we get only the cohomology classes which contain ``reachable" generators $\mathcal{O}_1$ and $\mathcal{O}_3$: 
\begin{itemize}
\item In dimension 1 (0 in the trace space): 1 generator $\emptyset$ (corresponding to the empty chain of obstacles)
    \item In dimension 2 (1 in the trace space): 2 generators $c'_1$, $c'_3$,
\item In dimension 3 (2 in the trace space): 1 generator $c'_1 \smile c'_3$
\end{itemize}

Indeed, the trace space between $u$ and $v$ can be shown to be homotopy equivalent to 
$S^1 \times S^1$, hence the resulting cohomology ring is the tensor product of 
$(R,R)$ with itself, giving $(R,R^2, R)$. 

We also note that the cohomology ring for points $u=(2,3,2)$ and $v=(4,4,4)$ is:
\begin{itemize}
\item In dimension 1 (0 in the trace space): 1 generator $\emptyset$ (corresponding to the empty chain of obstacles)
    \item In dimension 2 (1 in the trace space): 1 generator $c''_4$,
\end{itemize}
This corresponds to a circle $S^1$ in the trace space, as expected. 

{Finally, from $u=(1,1,1)$ to $v=(2,2,2)$ we get:}
\begin{itemize}
\item In dimension 1 (0 in the trace space): 1 generator (the empty chain of obstacles)
\item In dimension 2 (1 in the trace space): 2 generators ($c'''_2$ and $c'''_3$)
\end{itemize}
This corresponds to a wedge of circles $S^1$, as expected. 

Let us now exemplify the cohomological operation $\curvearrowright$ on $HM^*[X]^{(2,3,2)}_{(0,0,0)}\times HM^*[X]^{(4,4,4)}_{(2,3,2)}$. 
This gives elements of $HM^*[X]^{(4,4,4)}_{(0,0,0)}$ as follows:
\begin{itemize}
\item In dimension 1: 1 generator $\emptyset=\emptyset \curvearrowright \emptyset$
\item In dimension 2: $c'_1\curvearrowright c''_4=c_1\smile c_4$, $c'_1\curvearrowright c''_3=c_1\smile c_3$
\item In dimension 3: $(c'_1\smile c'_3)\curvearrowright c''_4=c_1\smile c_3\smile c_4$
\end{itemize}

These are still very specific cases in which the reachability relation (between obstacles) determines all cohomology classes. We claim that in more subtle situations, $\curvearrowright$ and $\smile$ exhibit more refined interactions. 
\end{example}

\bibliography{biblio}

\begin{thebibliography}{10}

\bibitem{albert}
Augustin Albert, Jérémy Dubut, and Eric Goubault.
\newblock Homological algebra in abelian framed bicategories: Exact sequences
  and embedding theorems, 2026.

\bibitem{Barychnikov}
Yuliy Baryshnikov.
\newblock Linear obstacles in linear systems, and ways to avoid them.
\newblock {\em Advances in Applied Mathematics}, 151:102579, 2023.

\bibitem{dubut2017}
J\'{e}r\'{e}my Dubut.
\newblock {\em Directed homotopy and homology theories for geometric models of
  true concurrency}.
\newblock PhD thesis, Universit\'{e} Paris-Saclay, 2017.

\bibitem{Dubut}
J{\'{e}}r{\'{e}}my Dubut, Eric Goubault, and Jean Goubault{-}Larrecq.
\newblock Directed homology theories and eilenberg-steenrod axioms.
\newblock {\em Appl. Categorical Struct.}, 25(5):775--807, 2017.

\bibitem{fahrenberg2004directed}
Ulrich Fahrenberg.
\newblock Directed homology.
\newblock {\em Electronic Notes in Theoretical Computer Science}, 100:111--125,
  2004.
\newblock CONCUR 2003: CMCIM and GETCO.

\bibitem{thebook}
Lisbeth Fajstrup, Eric Goubault, Emmanuel Haucourt, Samuel Mimram, and Martin
  Raussen.
\newblock {\em Directed Algebraic Topology and Concurrency}.
\newblock Springer, 2016.

\bibitem{book}
Lisbeth Fajstrup, Eric Goubault, Emmanuel Haucourt, Samuel Mimram, and Martin
  Raussen.
\newblock {\em Directed Algebraic Topology and Concurrency}.
\newblock Springer Publishing Company, Incorporated, 1st edition, 2016.

\bibitem{semi-ab}
Eric Goubault.
\newblock A semi-abelian approach to directed homology.
\newblock {\em Journal of Applied and Computational Topology}, 8(2):271--299,
  2024.

\bibitem{directedpersistencemodule}
Eric Goubault.
\newblock Directed homology and persistence modules.
\newblock {\em J. Appl. Comput. Topol.}, 9(1):3, 2025.

\bibitem{concur92}
Eric Goubault and Thomas~P. Jensen.
\newblock Homology of higher dimensional automata.
\newblock In {\em {CONCUR} '92, Third International Conference on Concurrency
  Theory, Stony Brook, NY, USA, August 24-27, 1992, Proceedings}, pages
  254--268, 1992.

\bibitem{eliot}
Eric Goubault and Eliot Médioni.
\newblock Homological algebra over non-unital rings and algebras, with
  applications to $(\infty, 1)$-categories, 2026.
\newblock In preparation, preprint available upon request.

\bibitem{grandis2004}
Marco Grandis.
\newblock Inequilogical spaces, directed homology and noncommutative geometry.
\newblock {\em Homology, Homotopy and Applications}, 6(1):413--437, 2004.

\bibitem{grandisbook}
Marco Grandis.
\newblock {\em Directed Algebraic Topology: Models of Non-Reversible Worlds}.
\newblock New Mathematical Monographs. Cambridge University Press, 2009.

\bibitem{patchkoria2006}
Alex Patchkoria.
\newblock On exactness of long exact sequences of homology semimodules.
\newblock {\em Journal of Homotopy and Related structures}, 1(1):229--243,
  2006.

\bibitem{RaussenII}
Martin Raussen.
\newblock Simplicial models for trace spaces {II}: {G}eneral higher dimensional
  automata.
\newblock {\em Algebr. Geom. Topol.}, 12(3), 2012.

\bibitem{Whitney}
Hassler Whitney.
\newblock On products in a complex.
\newblock {\em Annals of Mathematics}, 39(2), 1937.

\bibitem{cubechains}
Krzysztof Ziemiański.
\newblock Spaces of directed paths on pre-cubical sets, 2016.

\end{thebibliography}

\end{document}